\documentclass[11pt]{amsart}
\usepackage{a4, a4wide}
\usepackage{amssymb}
\usepackage{amsmath}
\usepackage{amsthm}
\usepackage{amstext}
\usepackage{amscd}
\usepackage{latexsym}
\usepackage{mathabx}
\usepackage{graphics}
\usepackage{color}
\usepackage[all]{xy}
\usepackage{enumitem}
\usepackage{color}
\input{xy}
\xyoption{all}

\usepackage[colorlinks, pagebackref]{hyperref}
\hypersetup{
  colorlinks=true,
  citecolor=blue,
  linkcolor=blue,
  urlcolor=blue}

\DeclareSymbolFont{symbolsC}{U}{ntxsyc}{m}{n}
\SetSymbolFont{symbolsC}{bold}{U}{ntxsyc}{b}{n}
\DeclareMathSymbol{\nsimeq}{\mathrel}{symbolsC}{59}

\makeatletter
\def\@tocline#1#2#3#4#5#6#7{\relax
  \ifnum #1>\c@tocdepth % then omit
  \else
    \par \addpenalty\@secpenalty\addvspace{#2}%
    \begingroup \hyphenpenalty\@M
    \@ifempty{#4}{%
      \@tempdima\csname r@tocindent\number#1\endcsname\relax
    }{%
      \@tempdima#4\relax
    }%
    \parindent\z@ \leftskip#3\relax \advance\leftskip\@tempdima\relax
    \rightskip\@pnumwidth plus4em \parfillskip-\@pnumwidth
    #5\leavevmode\hskip-\@tempdima
      \ifcase #1
      \or\or \hskip 2em \or \hskip 2homologyem \else \hskip 3em \fi%
      #6\nobreak\relax
    \dotfill\hbox to\@pnumwidth{\@tocpagenum{#7}}\par
    \nobreak
    \endgroup
  \fi}
\makeatother

\theoremstyle{plain}
\newtheorem{theorem}{Theorem}[section]
\newtheorem{introtheorem}{Theorem}[]

\newtheorem{lemma}[theorem]{Lemma}

\newtheorem{proposition}[theorem]{Proposition}

\theoremstyle{definition}

\newtheorem{notation}[theorem]{Notation}

\newtheorem{remark}[theorem]{Remark}

\newtheorem{definition}[theorem]{Definition}

\numberwithin{equation}{section}
\DeclareMathOperator*{\Hom}{Hom}
\DeclareMathOperator*{\Spec}{Spec}    % spectrum of a ring
\DeclareMathOperator*{\Pic}{Pic}      % Picard group
\DeclareMathOperator*{\Cl}{Cl}        % Class group of Weil divisors
\DeclareMathOperator*{\Proj}{Proj}    % Proj of a graded ring
\DeclareMathOperator*{\cont}{cont}	  % Content

\newcommand{\Sing}{{\rm Sing_*^{\#A^1}}} 
\newcommand{\Sp}{{\rm Sp}} 		

\def\<{\langle}
\def\>{\rangle} 
\def\-{\overline} 
\def\~{\widetilde}
\def\^{\widehat}

\def\@{\mathcal}
\def\!{\mathscr}
\def\#{\mathbb}
\def\_{\underline}

\mathchardef\mhyphen="2D

\begin{document}
\title{Naive $\#A^1$-homotopies on ruled surfaces}

\author{Chetan Balwe}
\address{Department of Mathematical Sciences, Indian Institute of Science Education and Research Mohali, Knowledge City, Sector-81, Mohali 140306, India.}
\email{cbalwe@iisermohali.ac.in}

\author{Anand Sawant}
\address{School of Mathematics, Tata Institute of Fundamental Research, Homi Bhabha Road, Colaba, Mumbai 400005, India.}
\email{asawant@math.tifr.res.in}
%\date{\today}
\date{}
\thanks{Chetan Balwe was supported by SERB-DST MATRICS Grant: MTR/2017/000690}
\thanks{Anand Sawant acknowledges support of the Department of Atomic Energy, Government of India, under project no. 12-R\&D-TFR-5.01-0500.  He was supported by DFG Individual Research Grant SA3350/1-1 while he was affiliated to Ludwig-Maximilians-Universit\"at, M\"unchen, where a part of this work was carried out.}
%\subjclass[2010]{14C15, 14C25, 19E15 (Primary)}
%\keywords{algebraic cycles; motivic cohomology; rost nilpotence}

\begin{abstract}
We explicitly describe the $\#A^1$-chain homotopy classes of morphisms from a smooth henselian local scheme into a smooth projective surface, which is birationally ruled over a curve of genus $> 0$.  We consequently determine the sheaf of naive $\#A^1$-connected components of such a surface and show that it does not agree with the sheaf of its genuine $\#A^1$-connected components when the surface is not a minimal model.  However, the sections of the sheaves of both naive and genuine $\#A^1$-connected components over schemes of dimension $\leq 1$ agree.  As a consequence, we show that the Morel-Voevodsky singular construction on a smooth projective surface, which is birationally ruled over a curve of genus $> 0$, is not $\#A^1$-local if the surface is not a minimal model.     
\end{abstract}

\maketitle

%\tableofcontents

\setlength{\parskip}{2pt plus2pt minus1pt}

\section{Introduction}

A basic question in $\#A^1$-homotopy theory is to determine the set of \emph{genuine $\#A^1$-homotopy classes} of morphisms from one scheme to another and study how they compare with the corresponding \emph{naive $\#A^1$-homotopy classes}.  Since the $\#A^1$-homotopy category $\@H(k)$ over a field $k$ is built out of the category of simplicial sheaves on the big Nisnevich site of smooth
schemes over $k$ by taking a suitable localization, one expects naive and genuine $\#A^1$-homotopy classes of morphisms from a smooth henselian local scheme $U$ into a smooth scheme $X$ to be closely related to one another.  

We will focus on smooth projective schemes $X$ over a base field $k$.  The genuine $\#A^1$-homotopy classes of morphisms from a smooth henselian local scheme $U$ into $X$ are just the sections $\pi_0^{\#A^1}(X)(U) = \Hom_{\@H(k)}(U,X)$ of the sheaf of \emph{$\#A^1$-connected components} of $X$ on $U$.  On the other hand, the naive $\#A^1$-homotopy classes $\@S(X)(U)$ of morphisms from a smooth henselian scheme $U$ into $X$ are the quotient of the set of morphisms of schemes $U \to X$ by the equivalence relation generated by naive $\#A^1$-homotopies, that is, morphisms $U \times \#A^1 \to X$.  If the Morel-Voevodsky singular construction $\Sing X$ is $\#A^1$-local (see Section \ref{section a1-connected components} for precise definitions), the sheaves $\@S(X)$ and $\pi_0^{\#A^1}(X)$ agree.  However, $\Sing X$ is $\#A^1$-local in very special cases (see \cite[Theorem 2.3.5, Remark 3.3.8]{Asok-Hoyois-Wendt-2} for a general result about $\#A^1$-locality of $\Sing$) and is not $\#A^1$-local in general (see \cite{Balwe-Hogadi-Sawant}, \cite{Balwe-Sawant-IMRN} for examples). 

It is easy to see that $\@S(X)(U)$ and $\pi_0^{\#A^1}(X)(U)$ always agree when $X$ has dimension $\leq 1$, and the same in case of a non-uniruled surface $X$ can be obtained from the dimension $\leq 1$ case (see \cite[Theorem 3.14, Corollary 3.15]{Balwe-Hogadi-Sawant}).  A result of Asok and Morel (\cite[Theorem 2.4.3]{Asok-Morel}; see also \cite[Theorem 3.9, Corollary 3.10]{Balwe-Hogadi-Sawant}) states that for a proper scheme $X$ over a field $k$, one has a bijection
\begin{equation}
\label{eqn Asok-Morel}
\@S(X)(\Spec F) \xrightarrow{\simeq} \pi_0^{\#A^1}(X)(\Spec F),
\end{equation}
for every finitely generated, separable field extension $F$ of $k$.  However, the corresponding result for a general $U$ cannot hold, as shown by the counterexample constructed in \cite[Section 4]{Balwe-Hogadi-Sawant} (for $U$ the spectrum of a henselian discrete valuation ring and $X$ a particular threefold).  

When the base field $k$ is algebraically closed of characteristic $0$, the set $\pi_0^{\#A^1}(X)(U)$ has been explicitly determined in \cite{Balwe-Sawant-ruled}, where $X$ is a smooth projective surface, birationally ruled over a curve of genus $>0$.  In this companion paper to \cite{Balwe-Sawant-ruled}, we determine the sheaf of naive $\#A^1$-connected components of a (birationally) ruled surface over an algebraically closed field of characteristic $0$.  The main result of the paper provides an extension of \eqref{eqn Asok-Morel} to sections over schemes of dimension $\leq 1$, when $X$ is a smooth projective surface.

\begin{introtheorem}
\label{theorem main theorem intro}
Let $k$ be an algebraically closed field of characteristic $0$ and let $X$ be a smooth projective surface over $k$, which is birationally ruled over a curve of genus $>0$.  Then the natural map
\[
\@S(X)(U) \to \pi_0^{\#A^1}(X)(U)
\]
is a bijection, for every henselian local scheme $U$ over $k$ of dimension $\leq 1$.
\end{introtheorem}

Theorem \ref{theorem main theorem intro} is known to hold in the case of non-uniruled surfaces by \cite[Corollary 3.15]{Balwe-Hogadi-Sawant} and the case of rational surfaces is easy to handle since they are covered by affine spaces \cite[Lemma 2.2.11]{Asok-Morel}.  Thus, Theorem \ref{theorem main theorem intro} in fact holds for all smooth projective surfaces over $k$.  It has been proved in \cite{Balwe-Sawant-ruled} that the sheaf $\pi_0^{\#A^1}(X)$ for such an $X$ agrees with the second iteration of the functor $\@S$ on $X$.  Theorem \ref{theorem main theorem intro} will be proved by directly obtaining a precise description of $\@S(X)(U)$ and comparing it with the description of $\@S^2(X)(U)$ obtained in \cite{Balwe-Sawant-ruled}.

We emphasize here that our method of proof of Theorem \ref{theorem main theorem intro} gives an explicit description of $\@S(X)(U)$ for any smooth henselian local $k$-scheme $U$ in terms of an algebraic condition, which can be easily verified.  The key result pertaining to this is Proposition \ref{proposition S single point}, which gives this description in the case of a single point blowup.  The general case reduces to this case, as outlined in Section \ref{section general case and applications} and is of independent interest.
 
As an application of the method of our proof, we get the non-$\#A^1$-locality of the Morel-Voevodsky singular construction applied to any smooth projective surface birationally ruled over a curve of genus $>0$, which is not a minimal model.

\begin{introtheorem}
\label{theorem application intro}
Let $k$ be an algebraically closed field of characteristic $0$ and let $X$ be a smooth projective surface over $k$, which is birationally ruled over a curve of genus $>0$.  If $X$ is not a minimal model, then the natural epimorphism 
\[
\@S(X) \to \@S^2(X) = \pi_0^{\#A^1}(X)
\]
is not an isomorphism of Nisnevich sheaves.  In particular, $\Sing X$ is not $\#A^1$-local.
\end{introtheorem}

The methodology used in the proofs of Theorems \ref{theorem main theorem intro} and \ref{theorem application intro} is to classify the naive $\#A^1$-homotopy classes of morphisms from a smooth henselian local $k$-scheme $U$ into $X$ birationally ruled over a curve $C$ of genus $>0$ by fixing a morphism $\gamma \colon U \to C$ and classifying homotopy classes of morphisms $U \to X$ over $\gamma$ by tracking where they map the closed point of $U$ into $X$.  Preliminaries about $\#A^1$-connectedness are given in Section \ref{section a1-connected components}, whereas results about local geometry of ruled surfaces are recalled in Section \ref{section preliminaries ruled}.  We first handle the case of the blowup of one point on the projectivization of a rank $2$ vector bundle on $C$ in Section \ref{section single point blowup}.  In Section \ref{section general case and applications}, the main results are obtained for a special class of blowups called \emph{nodal blowups} by using an appropriate modification of the arguments given in Section \ref{section single point blowup}.  One then reduces the general case to these blowups by employing the same method used in \cite{Balwe-Sawant-ruled}, which analyzes local structure of naive $\#A^1$-homotopy classes of blowups of such surfaces at infinitely near points.  We point out that although the strategy of the proofs of the main results is similar to the one in \cite{Balwe-Sawant-ruled}, the key geometric input in this paper (Proposition \ref{proposition S single point} and its proof) is very different in spirit and of independent interest to commutative algebraists and algebraic geometers.

\section{Naive and genuine \texorpdfstring{$\#A^1$}{A1}-connected components}
\label{section a1-connected components}

We fix a base field $k$ and denote by $Sm/k$ the category of smooth, finite-type, separated schemes over $\Spec k$.  We will denote by $\@H(k)$ the unstable $\#A^1$-homotopy category constructed by Morel and Voevodsky \cite{Morel-Voevodsky}.  One obtains $\@H(k)$ from the category of simplicial Nisnevich sheaves of sets on $Sm/k$ (also called \emph{spaces} over $k$ for brevity) by localizing with respect to Nisnevich local weak equivalences and further localizing with respect to the class of projection maps $\@X \times \#A^1 \to \@X$, for every simplicial Nisnevich sheaf of sets $\@X$.  All presheaves on $Sm/k$ will be extended to essentially smooth schemes (that is, schemes which are filtered inverse limits of diagrams of smooth schemes in which the transition maps are \'etale, affine morphisms) by defining $\@F(\varprojlim U_{\alpha}) = \varinjlim \@F(U_\alpha)$.  By a \emph{smooth henselian local scheme} over $k$, we will mean the henselization of the local ring at a smooth point of a scheme over $k$.

%A Nisnevich sheaf of sets $\@F$ on $Sm/k$ is \emph{$\#A^1$-invariant} if for every $U \in Sm/k$, the projection map $U \times \#A^1 \to U$ induces a bijection $\@F(U) \stackrel{\simeq}{\to} \@F(U \times \#A^1)$.  A scheme $X$ over $k$ is called \emph{$\#A^1$-rigid} if the associated Nisnevich sheaf given by its functor of points is $\#A^1$-invariant.

We will follow the conventions and notation of \cite{Balwe-Hogadi-Sawant} and \cite{Balwe-Sawant-ruled} regarding $\#A^1$-connected components.  We briefly recall the notions relevant to this paper below.

%%%%%%%%%%%%%%%%%%%%%%%%%%%%%%%%

\begin{definition}
\label{definition sing}
Let $\mathcal X$ be a space over $k$. Define ${\rm Sing_*^{\mathbb A^1}} \mathcal X$ to be the simplicial sheaf given by
\[
({\rm Sing_*^{\mathbb A^1}} \mathcal X)_n = \underline{\rm Hom}(\Delta^n,\mathcal X_n), 
\]
\noindent where $\Delta^{\bullet}$ denotes the cosimplicial sheaf  
\[
\Delta^n  = {\rm Spec}~ \frac{k[x_0,...,x_n]}{\<\sum_ix_i=1\>}
\]
\noindent with natural face and degeneracy maps analogous to the ones on topological simplices.
\end{definition}

Given a simplicial sheaf of sets $\mathcal X$ on $Sm/k$, we will denote by $\pi_0(\mathcal X)$ the presheaf on $Sm/k$ that associates with $U \in Sm/k$ the coequalizer of the diagram $\mathcal X_1(U) \rightrightarrows \mathcal X_0(U)$, where the maps are the face maps coming from the simplicial data of $\mathcal X$.  We will denote the Nisnevich sheafification functor by $a_{\rm Nis}$.

\begin{definition}
\label{definition S}
The sheaf of \emph{$\mathbb A^1$-chain connected components} of a space $\mathcal X$ is defined to be $$\mathcal S(\mathcal X) := a_{\rm Nis} \left(\pi_0{\rm \Sing} \mathcal X \right).$$
\end{definition}

Thus, when $\@X$ is the space given by a scheme $X$ over $k$, for every smooth henselian local $k$-scheme $U$, the set $\@S(X)(U)$ is the quotient of $X(U)$ by the equivalence relation generated by $\sigma_0^*(h) = \sigma_1^*(h)$, where $h \in X(U \times \#A^1)$ is a naive $\#A^1$-homotopy of $U$ into $X$ and $\sigma_i\colon  U \to U \times \#A^1$ denotes the $i$-section.

\begin{definition}
\label{definition homotopy}
Let $X \in Sm/k$ and let $U$ be an essentially smooth scheme over $k$.
\begin{enumerate}[label=$(\arabic*)$]
\item An \emph{$\#A^1$-homotopy} of $U$ in $X$ is a morphism $h\colon  U \times \#A^1_{k} \to X$. We say that $f,g \in X(U)$ are \emph{$\#A^1$-homotopic} if there exists an $\#A^1$-homotopy $h\colon  U \times \#A^1 \to X$ such that $\sigma_0^*(h) = f$ and $\sigma_1^*(h)= g$.

\item An \emph{$\#A^1$-chain homotopy} of $U$ in $X$ is a finite sequence $\@H=(h_1, \ldots, h_r)$ where each $h_i$ is an $\#A^1$-homotopy of $U$ in $X$ such that $\sigma_1^*(h_i) = \sigma_0^*(h_{i+1})$ for $1 \leq i \leq r-1$. We say that $f, g \in X(U)$ are \emph{$\#A^1$-chain homotopic} if there exists an $\#A^1$-chain homotopy $h=(h_1, \ldots, h_r)$ such that $\sigma_0^*(h_1) = f$ and $\sigma_1^*(h_r) = g$.

\item The \emph{total space} $\Sp(h)$ of an $\#A^1$-homotopy $h\colon  U \times \#A^1_{k} \to X$ of $U$ in $X$ is just $U \times \#A^1_{k}$.  The \emph{total space} $\Sp(\@H)$ of an $\#A^1$-chain homotopy $\@H=(h_1, \ldots, h_r)$ is the disjoint union of the total spaces of $h_1, \ldots, h_r$.  We write $f_{\@H}\colon  \Sp(\@H) \to U$ for the canonical projection on $U$.  The morphism $h_{\@H}\colon  \Sp(\@H) \to X$ is just the morphism defining the chain homotopy.
\end{enumerate} 
\end{definition}

\begin{definition}
\label{definition pi0A1}
The sheaf of \emph{$\mathbb A^1$-connected components} of a space $\mathcal X$ is defined to be 
\[
\pi_0^{\#A^1}(\mathcal X) = a_{\rm Nis} {\Hom}_{\@H(k)}(-, \mathcal X).
\]
Equivalently, $\pi_0^{\#A^1}(\mathcal X) = a_{\rm Nis} \left(\pi_0{L_{\mathbb A^1}} \mathcal X \right)$, where $L_{\mathbb A^1}$ denotes an $\#A^1$-fibrant replacement functor for the injective Nisnevich local model structure.
\end{definition}

Given schemes $U$ and $X$ over $k$, two elements of $X(U)$ are said to be \emph{naively $\#A^1$-homotopic} if they are $\#A^1$-chain homotopic.  Two elements of $X(U)$ are said to be \emph{genuinely $\#A^1$-homotopic} if they map to the same element in $\pi_0^{\#A^1}(X)(U)$.

\begin{remark}
\label{remark sing a1-local}
If $\Sing \@X$ is $\#A^1$-local in the sense of \cite[Definition A.2.1]{Asok-Morel}, then the natural epimorphism $\@S(\@X) \to \pi_0^{\#A^1}(\@X)$ is an isomorphism (see \cite[Remark 2.2.9]{Asok-Morel}).
\end{remark}

\section{Preliminaries on ruled surfaces}
\label{section preliminaries ruled}

Theorem \ref{theorem main theorem intro} and Theorem \ref{theorem application intro} will be proved by using the classification of ruled surfaces and an induction on the \emph{complexity} of the blowup involved.  In this section, we will briefly outline the formalism developed in \cite{Balwe-Sawant-ruled} to handle blowups of ruled surfaces and state some results about local geometry of ruled surfaces proved in \cite[Section 4]{Balwe-Sawant-ruled}.  We follow the conventions of \cite[Section 4]{Balwe-Sawant-ruled}, where we also refer the reader to for more details.  We will focus on smooth projective surfaces, which are obtained from $\#P^1 \times C$ by a finite number (possibly zero) of successive blowups at smooth, closed points, where $C$ is the henselization of the local ring of a closed point on a smooth projective curve of genus $\geq 1$.  The results for general ruled surfaces will follow from the ones in this special case by standard arguments (see the proofs of the main results in Section \ref{section general case and applications}).

Throughout the rest of the article, we will assume that $k$ is an algebraically closed field of characteristic $0$, although this assumption is not always necessary for some of the results proved in the rest of the article.  We fix the following setting for this section.

\begin{notation}
\label{notation nodal blowups}
\hspace{1cm}
\begin{enumerate}
\item Let $x$ be a variable and let $A$ be the henselization of the polynomial ring $k[x]$ at the maximal ideal $\<x\>$. Let $C = \Spec A$ and let $c_0$ denote the closed point of $C$.  Every $1$-dimensional regular henselian local ring containing $k$ is isomorphic to $A$.  

\item Let $Y$ and $Z$ be variables which denote the homogeneous coordinates on $\#P^1_C$. Thus, $\#P^1_C = \Proj A[Y,Z]$. 

\item Let $y$ denote the rational function $Y/Z$ on $\#P^1_C$ which is defined on the open subscheme of $\#P^1_C$ defined by the condition $Z \neq 0$. This open subscheme is simply $\Spec A[y] \simeq \#A^1_C$ and the point $(c_0, [0:1])$ is defined by the ideal $\<x,y\>$ of $A[y]$.  

\item Let $\ell_{\infty}$ denote the closed subscheme $C \times \{[0:1]\}$ of $\#P^1_C$, which is the divisor of zeros of $y$. The closed subscheme $C \times \{[1:0]\}$ is the divisor of poles of $y$ and will be denoted by $\ell_{-\infty}$. 
\end{enumerate}
\end{notation}

\begin{definition}
\label{definition nodal blowups}
Let $X$ be any scheme that is obtained from $\#P^1_C$ by a finite number (possibly zero) of successive blowups at smooth, closed points.
\begin{enumerate}
\item The preimage of $c_0$ under the morphism $X \to C$ is a connected scheme, the irreducible components of which are isomorphic to $\#P^1_k$. We refer to these as \emph{lines} on $X$.

\item We define a \emph{pseudo-line} to be any curve that is either a line on $X$ or the proper transform of $\ell_{\infty}$.

\item A \emph{node} on $X$ is defined to be the intersection point of two pseudo-lines. It is easy to see that any node is the point of intersection of exactly two pseudo-lines. 
%(The curve $\ell_{-\infty}$ is not being considered as a pseudo-line because the nodes lying on its proper transforms are not relevant to our discussion.)
Thus, for instance, the point $(c_0, [0:1])$ is the only node on the scheme $\#P^1_C$. 

\item We will denote by $\@N$ the collection of schemes $X$ admitting a morphism $X \to \#P^1_C$, which factors as  
\[
X = X_r \xrightarrow{\pi_r} X_{r-1} \xrightarrow{\pi_{r-1}} \cdots \xrightarrow{\pi_1} X_0 = \#P^1_C
\]
with $r \geq 0$, where for all $i\geq 1$, $\pi_i\colon  X_i \to X_{i-1}$ is the blowup of $X_{i-1}$ at some \emph{node} of $X_{i-1}$.  Elements of $\@N$ will be referred to as \emph{nodal blowups}.
\end{enumerate}
\end{definition}

Since all the nodes in all the $X \in \@N$ lie over the point $(c_0, [0:1])$, we see that the ideals defining these nodes are generated by rational functions in $x$ and $y$.  We record the following lemma from \cite[Lemma 4.12]{Balwe-Sawant-ruled}.

% For any $X \in \@N$, any line on $X$ has a parameter of the form $x^a/y^b$ for non-negative integers $a$ and $b$. Observe that the integers $a$ and $b$ are uniquely determined by this line. Indeed, if $x^a/y^b$ and $x^c/y^d$ are parameters on the same line, they are related by a fractional linear transformation. In other words, there exist $\alpha, \beta, \gamma, \delta$ in $k$ such that $\alpha \delta - \beta \gamma \neq 0$ such that 
% \[
% x^a/y^b = \frac{\alpha (x^c/y^d) + \beta}{\gamma(x^c/y^d) + \delta}. 
% \]
% As $x$ and $y$ are algebraically independent, it is easy to prove that this can only happen if $(a,b) = (c,d)$. 

\begin{lemma}
\label{lemma blowups description}
Let $X \in \@N$. Then:
\begin{itemize}
\item[(a)] If a line $l$ on $X$ has a parameter of the form $x^a/y^b$, then $a$ and $b$ are coprime non-negative integers.  
\item[(b)] If $C_1$ and $C_2$ are pseudo-lines on $X$ with parameters $x^a/y^b$ and $x^c/y^d$ which intersect at a node defined by the ideal $\<x^a/y^b, y^d/x^c\>$, then $ad-bc = 1$ (and so, in particular, $a/b > c/d$). 
\end{itemize}
\end{lemma}

This shows that for pair of coprime integers $(a,b)$ with $a<b$, there is at most one line on $X$ with parameter $x^a/y^b$.  It can be shown that if such lines exist on two schemes $X_1$, $X_2$ in $\@N$, then the canonical birational maps from $X_1$ to $X_2$ will take the lines into each other.  If $X \in \@N$ has $n$ lines, we can associate to it the sequence of rational numbers $r_0 = -\infty < r_1 = 0 < \ldots < r_n = 1 < r_{n+1} = \infty$ such that:
\begin{itemize}
\item For $1 \leq i \leq n$, if $r_i$ is written in reduced fractional form as $a_i/b_i$ for some non-negative integers $a_i$, $b_i$, then $X$ contains a line parameterized by $x^{a_i}/y^{b_i}$. We will label this line as $\ell_{r_i}$. For $i = 0$ and $i=n+1$, the symbol $\ell_{r_i}$ will denote the pseudo-lines $\ell_{-\infty}$ and $\ell_{\infty}$ respectively.   
\item For $1 \leq i \leq n$, the line $\ell_{r_i}$ only meets the pseudo-lines $\ell_{r_{i-1}}$ and $\ell_{r_{i+1}}$. If $r_i$ and $r_{i+1}$ are written in reduced fractional form as $r_i = a_i/b_i$ and $r_{i+1} = a_{i+1}/b_{i+1}$ then they meet in a node defined by the maximal ideal $\<x^{a_{i+1}}/y^{b_{i+1}}, y^{b_i}/x^{a_i}\>$. (Here, if $r_{i+1} = \infty$, we choose $a_{i+1} = 1$ and $b_{i+1} = 0$.)
\end{itemize}

\section{The case of the blow-up of a minimal model at a single closed point}
\label{section single point blowup}

The aim of this section is to prove the following result, which proves both Theorem \ref{theorem main theorem intro} and \ref{theorem application intro} stated in the introduction in a special case of the blowup of a single closed point on $\#P^1 \times C$, where $C$ is a curve of genus $>0$.

\begin{theorem}
\label{theorem single point}
Let $E$ be a $\#P^1$-bundle over a smooth projective curve $C$ of genus $>0$. Let $X$ be the blowup of $E$ at a single closed point. Then 
\begin{enumerate}[label=$(\alph*)$]
\item $\@S(X)(U) = \@S^2(X)(U)$, where $U = \Spec R$, for a henselian discrete valuation ring $R$ over $k$;
\item $\@S(X) \neq \@S^2(X)$.
\end{enumerate}
\end{theorem}

Along the course of proving Theorem \ref{theorem single point}, we will give a complete description of $\#A^1$-chain homotopy classes of a smooth henselian local scheme $U$ into $X$.

Fix a smooth henselian local ring $(R, \mathfrak{m})$, $U = \Spec R$. We fix a morphism $\gamma\colon  U \to C$ and consider pullbacks of $X$ and $E$ with respect to $\gamma$ which we denote by $X_{\gamma}$ and $E_{\gamma}$. Thus $X_{\gamma}$ is obtained by blowing up a closed subscheme of $E_{\gamma}$. But since a $\#P^1$-bundle is \'etale locally trivial, we see that $E_{\gamma} \simeq \#P^1_U$. Thus we may assume that $E \to C$ is the trivial $\#P^1$-bundle $\#P^1_C$. 

Thus, let us assume that $E  = \#P^1_C = \#P^1_k \times C$ and $X$ has been obtained from $\#P^1_C$ by blowing up the point $p:= (c_0, (0:1))$ where $c_0$ is a closed point in $C$. Suppose the image of $\gamma\colon  U \to C$ does not contain $c_0$. Then $X_{\gamma}$ is isomorphic to $\#P^1_U$. On the other hand, if $\gamma$ maps the whole of $U$ into the point $c_0$, then $X_{\gamma}  = U \times T$ where $T$ consists of two copies of $\#P^1_k$ intersecting transversely in a single point. In both these cases, any two sections of $X_{\gamma} \to U$ are $\#A^1$-chain homotopic and thus the required result follows immediately. Thus we may now focus on the case in which $\gamma$ maps the closed point of $U$ to $c_0$ but does not map the whole of $U$ into $c_0$.

Let $\omega$ be a uniformizing parameter in the ring $\@O_{C,c_0}$. Then $X_{\gamma}$ is obtained from $\#P^1_U = \Proj R[Y,Z]$ by blowing up the closed subscheme corresponding to the homogeneous ideal $\< \gamma^*(\omega), Y\>$. The remainder of this section will be devoted to analyzing the $\#A^1$ chain homotopy classes in $X_{\gamma}(U)$, where we view $X_{\gamma}$ as a $U$-scheme (thus, the projections of all the $\#A^1$-chain homotopies to $U$ must be constant).  We now set up some notation to analyze the scheme $X_{\gamma}$.

\begin{notation}
\label{notation blowup}
The following notation will stay in effect for the remainder of this section:
\begin{enumerate}[label=$(\alph*)$]
\item Let $R$ denote the henselization of a ring of the form $\@O_{V,v}$ where $V$ is a smooth scheme over $k$ and $v$ is a point of $V$. Let $\mathfrak{m}$ be the maximal ideal of $R$ and let $\kappa = R/\mathfrak{m}$. We will denote $\Spec R$ by $U$ and the closed point of $U$ by $u$. For $r \in R$, $U(r)$ will denote the closed subscheme of $r$ corresponding to the ideal $\<r\>$. 

\item $Y, Z$ will denote the homogeneous coordinates on $\#P^1_R = \Proj R[Y,Z]$. 

\item For any $r_0 \in \mathfrak{m} \backslash \{0\}$, we define
\begin{eqnarray*}
I(r_0) & :=  &  \text{ideal sheaf on $\#P^1_U$ corresponding to the}\\
        &    &  \text{homogeneous ideal } \<r_0, Y\> \text{.}
\end{eqnarray*}
Let $T(r_0)$ denote the closed subscheme of $\#P^1_U$ corresponding to this ideal sheaf. Let $\theta_{r_0}\colon  X(r_0) \to \#P^1_U$ denote the blowup of $\#P^1_U$ at $T(r_0)$.
\end{enumerate} 
\end{notation}

We wish to examine the sections of morphisms $X(r_0) \to U$. The scheme $\#P^1_U$ is covered by affine patches isomorphic to $\#A^1_U$ given by the conditions $Y \neq 0$ and $Z \neq 0$. We will denote these by $\Spec R[Y/Z]$ and $\Spec R[Z/Y]$ respectively. Using the fact that $U$ is a local scheme it is easy to see that any section of $\#P^1_U$ factors through $\Spec R[Y/Z]$ or $\Spec R[Z/Y]$. The sections of $\Spec R[Y/Z] \to U$ are given by $R$-algebra homomorphisms $R[Y/Z] \to R$. Such a homomorphism is determined by a choice for the image of $Y/Z$. 

\begin{notation}
We denote by $\beta_r\colon  U \to \Spec R[Y/Z] \to \#P^1_U$ the section of $\#P^1_U \to U$ determined by $Y/Z \mapsto r$. Similarly, we denote by $\alpha_r\colon  U \to \Spec R[Z/Y] \to \#P^1_U$ the section of $\#P^1_U \to U$ determined by $Z/Y \mapsto r$. Thus we see that if $r$ is a unit in $R$, then $\alpha_r = \beta_{1/r}$.
\end{notation}

A section of $X(r_0) \to U$ can be composed with $X(r_0) \to \#P^1_U$ to give a section of $\#P^1_U \to U$. Of course, not every section of $\#P^1_U \to U$ can be lifted to give a section of $X(r_0) \to U$.  Since $r_0 \neq 0$, any section of $\#P^1_U \to U$ maps the generic point of $U$ outside the scheme $T(r_0)$.  As $X(r_0) \to \#P^1_U$ is an isomorphism outside $T(r_0)$, it follows that if a lift of a section of $\#P^1_U \to U$ to $X(r_0)$ exists, then it must be unique. 

\begin{notation}
Let $r_0 \in \mathfrak{m} \backslash \{0\}$. For $r \in R$, if a lift of $\alpha_r\colon  U \to \#P^1_U$ to $X(r_0)$ exists, it will be denoted by $\alpha_r^{r_0}$. Similarly, we define  $\beta^{r_0}_r$ to be the lift of $\beta_r$ to $X(r_0)$, if it exists. \end{notation}

A section $\alpha\colon  U \to \#P^1_U$ of $\#P^1_U \to U$, lifts to $X(r_0)$ if and only if the ideal sheaf $\alpha^*(I(r_0))$ is principal. Since $U$ is local, this ideal sheaf corresponds to an ideal of $R$. Unless $\alpha$ is of the form $\beta_r$ where $r \in \mathfrak{m}$, we have $\alpha^*(I(r_0)) = \@O_{U}$. When $\alpha = \beta_r$ for $r \in \mathfrak{m}$, the ideal sheaf $\alpha^*(I(r_0))$ corresponds to the ideal $\<r_0,r\>$ in $R$. 

\begin{lemma}
\label{lemma lifts to X(r_0)}
Let $r_0 \in \mathfrak{m} \backslash \{0\}$. Let $\alpha\colon  U \to \#P^1_U$ be a section of $\#P^1_U \to U$ which admits a lift to $X(r_0)$. Then one of the following must hold:
\begin{enumerate}[label=$(\alph*)$]
\item $\alpha^*(I(r_0)) = \<1\>$. In this case, $\alpha$ is of the form $\alpha_r$ where $r \in R$. 

\item $\alpha^*(I(r_0)) = \<r_0\>$. In this case, $\alpha$ is of the form $\beta_r$ where $r_0|r$. 

\item $\alpha^*(I(r_0)) = \<r\>$ for some $r \in R$ such that $r|r_0$ but $r_0 \notdivides r$. In this case, $\alpha$ is of the form $\beta_{r^{\prime}}$ where $r^{\prime}/r$ is a unit in $R$.
\end{enumerate}
\end{lemma}

\begin{proof}
We have $\alpha^*(I(r_0)) = \<1\>$ if and only if $\alpha$ factors through the complement of the scheme $T(r_0)$ of $\#P^1_U$. This is so if and only if $\alpha$ is of the form $\alpha_r$ where $r \in R$. This proves (a). 

We now have to consider the case that $\alpha$ is of the form $\beta_r\colon  U \to \#P^1_U$ for some $r \in \mathfrak{m}$. In this case, $\alpha^*(I(r_0)) = \<r_0, r\>$. Since $R$ is a local domain, this ideal is equal to $\<r_0\>$ if $r_0|r$ or is equal to $\<r\>$ if $r|r_0$. This proves (b) and (c).
\end{proof}

\begin{notation}
Let $r_0 \in \mathfrak{m} \backslash \{0\}$. Then for $r \in R$, $\@A_r^{r_0}$ denotes the set of sections $\alpha\colon  U \to \#P^1_U$ such that the ideal sheaf $\alpha^*(I(r_0))$ corresponds to the ideal $\<r\>$ in $R$. By Lemma \ref{lemma lifts to X(r_0)}, $\@A^{r_0}_r$ is non-empty if and only if $r|r_0$.
\end{notation}

\begin{remark}
\label{remark image of closed point}
We note a geometric interpretation of the families $\@A^{r_0}_r$. Observe that the preimage of $c_0$ under the morphism $X \to C$ is the union of two lines $L_1$ and $L_2$. Here $L_1$ is the proper transform of the line $\{c_0\} \times \#P^1_{k} \subset \#P^1_C$ and $L_2$ is the exceptional divisor.  A section $\alpha\colon  U \to \#P^1_U$ maps the closed point $u$ of $U$ into $L_2$ if and only if $r$ is not a unit.  It maps $u$ into $L_1$ if and only if $r|r_0$ and $r_0 \notdivides r$.

Thus, the elements of $\@A^{r_0}_1$ correspond to those lifts of $\gamma$ which map $u$ into $L_1  \backslash L_2$. The elements of $\@A^{r_0}_{r_0}$ are those sections which map $u$ into $L_2 \backslash L_1$. The set of sections which map $u$ to the point $L_1 \cap L_2$ is precisely the union of the sets $\@A^{r_0}_r$ where $r$ is not a unit, $r|r_0$ and $r_0 \notdivides r$. 

This indicates a certain symmetry between the sets $\@A^{r_0}_1$ and $\@A^{r_0}_{r_0}$. Indeed, these two sets can be transformed into each other by means of an \emph{elementary transformation}.
\end{remark}

We will use the following special case of \cite[Proposition 3.7]{Balwe-Sawant-ruled}.

\begin{lemma}
\label{lemma local generator ghost homotopy}
Let $r_0 \in \mathfrak{m} \backslash \{0\}$.  Let $\alpha$ and $\alpha^{\prime}$ be sections of $\#P^1_U \to U$ which are connected by an $\#A^1$-chain homotopy $\@H$. Then there exists $r \in R$ such that $r|r_0$ and $\alpha, \alpha^{\prime} \in \@A^{r_0}_r$.  Moreover, in this case $h_{\@H}^*(I(r_0))$ is generated by $f^*_{\@H}(r)$. 
\end{lemma}

Lemma \ref{lemma local generator ghost homotopy} gives us a necessary condition for two elements $\beta_{r_1}$ and $\beta_{r_2}$ to be $\#A^1$-homotopic. However, we will see now that it is not strong enough. The following two results give a complete description of $\#A^1$-homotopy classes of the sections of $X(r_0) \to U$. 

\begin{lemma}
\label{lemma easy A1-homotopy classes}
Let $r_0 \in \mathfrak{m} \backslash \{0\}$. Any two elements of $\@A^{r_0}_r$ are $\#A^1$-homotopic via homotopies that lift to $X(r_0)$ if $r = 1$ or $r_0$.  
\end{lemma}

\begin{proof}
First consider the case $r = 1$. Suppose $\alpha$, $\alpha^{\prime}$ are in $\@A^{r_0}_1$. Then it is easy to see that $\alpha = \alpha_{r_1}$ and $\alpha^{\prime} = \alpha_{r_2}$ for some units $r_1, r_2 \in R^{\times}$. Then these two elements are homotopic in $\#P^1_U$ via the homotopy given by the $R$-algebra homomorphism $$R[Z/Y] \to R[T], \quad Z/Y \mapsto r_1^{-1}(1-T) + r_2^{-1}T.$$  Note that we are using the parameter $Z/Y$ instead of $Y/Z$, so $\alpha_{r_i}^*(Z/Y) = r_i^{-1}$ for $i=1,2$.  This defines a morphism $\#A^1_U \to \#P^1_U$, which factors through the open immersion $\#P^1_U \backslash T(r_0) \hookrightarrow \#P^1_U$ and hence, lifts to $X(r_0)$. 

The case $r=r_0$ can be deduced from the case $r = 1$ by means of an elementary transformation (as mentioned in Remark \ref{remark image of closed point}). However, we will simply write out an explicit homotopy. If $\alpha$, $\alpha^{\prime}$ are in $\@A^{r_0}_{r_0}$, they are of the form $\beta_{r_1}$ and $\beta_{r_2}$ where $r_0|r_1$ and $r_0|r_2$. Then consider the homotopy $h\colon  \#A^1_U \to \#P^1_U$ given by the $R$-algebra homomorphism $R[Y/Z] \to R[T]$, $Y/Z \mapsto r_1(1-T) + r_2T$. Then $h^*(I(r_0))$ is the sheaf corresponding to the ideal $\<r_0, r_1(1-T) + r_2T\> = \<r_0\>$ which is principal. Thus $h$ lifts to $X(r_0)$ as claimed. 
\end{proof}

\begin{proposition}
\label{proposition S single point}
Let $r_0 \in \mathfrak{m} \backslash \{0\}$. Suppose $r_1, r_2 \in R$ are non-units such that $r_i | r_0$ but $r_0 \notdivides r_i$ for $i=1,2$. Let $r' = r_0/r_1$. Then $\beta_{r_i}^{r_0}$ for $i = 1,2$ are $\#A^1$-chain homotopic in $X(r_0)$ if and only if the following two conditions hold:
\begin{enumerate}[label=$(\alph*)$]
\item $r_1/r_2$ is a unit. 

\item $r_1/r_2 -1 \in \sqrt{\<r_1\>} + \sqrt{\<r'\>} \subset \sqrt{\<r_1, r'\>}$. 
\end{enumerate}
\end{proposition}

\begin{proof}
We first prove the necessity of the conditions.  Clearly, it suffices to treat the case where $\beta_{r_1}^{r_0}$ and $\beta_{r_2}^{r_0}$ are $\#A^1$-homotopic in $X(r_0)$.  

The fact that $r_1/r_2$ is a unit is already known to us by Lemma \ref{lemma local generator ghost homotopy}. An $\#A^1$-homotopy $h\colon  \#A^1_U \to \#P^1_U$ is given by the choice of an invertible sheaf on $\#A^1_U$ along with the choice of two generating sections. Since $U$ is essentially smooth, we have $\Pic(U) \simeq \Cl(U)$ and $\Pic(\#A^1_U) \simeq \Cl(\#A^1_U)$, where $\Cl$ denotes the group of Weil divisors.  The projection $\#A^1_U \to U$ induces an isomorphism $\Cl(U) \simeq \Cl(\#A^1_U)$.  Since $U$ is local, this implies that any invertible sheaf on $\#A^1_U$ is isomorphic to $\@O_{\#A^1_U}$. Thus we see that the homotopy $h$ can be given by two generating global sections of $\@O_{\#A^1_U}$, that is, two polynomials $p(T), q(T) \in R[T]$ such that $\<p(T),q(T)\>$ is the unit ideal. If $D(q)$ denotes the open subscheme of $\#A^1_U$ where $q$ is a unit, then $h$ maps $D(q)$ into the open subscheme $\#A^1_U = \Spec R[Y/Z] \subset \#P^1_U$ and the 
morphism $h|_{D(q)}$ is given by the $R$-algebra homomorphism $Y/Z \mapsto p(T)/q(T)$.  Notice that $p(0)/q(0) = r_1$ and $p(1)/q(1) = r_2$. By the first equality, we see that $r_1|p(0)$ and thus, in particular, $p(0)$ is a non-unit. However, $\<p(0),q(0)\>$ is the unit ideal in $R$ (since $\<p,q\>$ is the unit ideal in $R[T]$). Thus $q(0)$ is a unit in $R$. Thus $q(T)$ is a primitive polynomial, that is, its content $\cont(q)$, is a unit. 

Now suppose that the ideal sheaf $h^*(I(r_0))$ is locally principal. Since $R[T,q^{-1}]$ is a unique factorization domain, the ideal $h^*(I(r_0))(D(q))$ is principal.  By the formula defining $h$, this ideal is equal to $\<r_0, p/q\>$.  Since $h$ lifts to $X(r_0)$, we know that this ideal is actually equal to $r_1R[T,q^{-1}]$. Therefore, $$p \in r_1 R[T,q^{-1}] \cap R[T] = \bigcup_{n \geq 0}(r_1: q^n).$$  We claim that $(r_1:q^n) = \<r_1\>$ for every $n \geq 0$. Indeed, suppose $f \in (r_1:q^n)$. If $n = 0$, there is nothing to prove. So, we assume $n>0$. There exists an element $g(T) \in R[T]$ and an integer $n$ such that $r_1g = fq^n$ in $R[T]$. But then, $q^n|r_1g$ in $R[T]$. As $q$ is a primitive polynomial, $q^{n}| (r_1g/\cont(r_1g)) = g/\cont(g)$. In particular, $q^n|g$ and thus $r_1|f$ in $R[T]$, as claimed. This shows that $r_1|p$ in $R[T]$ and we write $p = r_1p'$ where $p' \in R[T]$.

The ideal $\<r_1p', q\>$ is a unit ideal. Thus, modulo $r_1$, we see that the polynomial $q(T)$ must be a unit. Hence, if $q(T) = q_0 + q_1T + \ldots$, we see that $q_0$ is a unit (as we have already seen before) and $q_i$ is nilpotent modulo $r_1$ for $i \geq 1$. 

In $R[T,q^{-1}]$, we have $\<r_0,r_1p^{\prime}/q\> = \<r_1\>$, which implies that $\<r^{\prime},p^{\prime}/q\> = \<1\>$.  Hence, there exists an integer $n \geq 0$ such that $q^n$ lies in the ideal $\<r^{\prime}, p^{\prime}\>$ in $R[T]$.  Since $\<p,q\> = 1$, we have $\<p^n,q^n\> = 1$, which implies that $\<(p^{\prime})^n,q^n\>=1$. Thus, $\<r^{\prime}, p^{\prime}\>$, which contains $\<(p^{\prime})^n,q^n\>$ is equal to the unit ideal. Thus if $p^{\prime}(T) = p_0^{\prime} + p_1^{\prime}T + \ldots$, we see that $p_0^{\prime}$ is a unit and $p_i^{\prime}$ is nilpotent modulo $r^{\prime}$ for $i \geq 1$. 

Using $p(1)/q(1) = r_2$, we see that $$r_2 = r_1\left( \dfrac{u_1 + t_1}{u_2 + t_2} \right),$$ where $u_i$ are units, $t_1$ is nilpotent modulo $r^{\prime}$ and $t_2$ is nilpotent modulo $r_1$. Using $p(0)/q(0) = r_1$, we see that $u_1/u_2 = 1$. Thus, we may assume that $u_1 = u_2 = 1$. Let $$t_3 = -\dfrac{t_2}{1 + t_2}.$$  Then we see that $t_3$ is nilpotent modulo $r_1$ and $r_2 = r_1(1 + t_1)(1 + t_3)$. Thus $$\dfrac{r_2}{r_1} - 1 = t_1 + t_3 + t_1t_3$$ and it is obvious that this element lies in $\sqrt{\<r_1\>} + \sqrt{\<r^{\prime}\>}$ as claimed. This proves the necessity of the conditions. 

We now prove the sufficiency. Let $s_1$ and $s^{\prime}$ be the square-free parts of $r_1$ and $r^{\prime}:= r_0/r_1$. Then $\sqrt{\<r_1\>} = \<s_1\>$ and $\sqrt{\<r'\>} =\<s'\>$. By $(b)$, we have $r_2/r_1 -1 = s_1 \delta + s' \delta'$, for some $\delta, \delta' \in R$.  We define $r_3 = r_1(1 + s_1 \delta_1)$. 

Consider the rational function $$f_1(T) = \dfrac{r_3}{1 + s_1 \delta_1 T}.$$ The ideal $\<r_3, 1 + s_1 \delta_1 T\>$ is the unit ideal since $s_1 \delta_1$ is nilpotent modulo $r_1$ and consequently, nilpotent modulo $r_3$. Thus $f_1$ defines a homotopy 
\[
h_1\colon  \#A^1_U \to \#P^1_U; Y/Z \mapsto \dfrac{r_3}{1 + s_1 \delta_1 T}.
\]
Since $f_1(0) = r_3$ and $f_1(1) = r_1$, this homotopy connects $\beta_{r_3}$ to $\beta_{r_1}$. Now, let $p$ be a point of $\#A^1_U$ and let $z \in U$ denote its image under the projection $\#A^1_U \to U$. We see that the ideal 
\[
h_1^*(I(r_0))_p = \left\<r_0, \frac{r_3 }{1 + s_1 \delta_1T} \right\>_p \subset R[T]_p 
\]
is principal if $1+s_1 \delta_1T$ is a unit at $p$, since $r_3|r_0$. If $1 + s_1 \delta_1 T$ is not a unit, then $s_1$, $\delta_1$ and $T$ are all units at $p$. Thus, $r_1$ is also a unit at $p$. Since $r_3$ is a unit multiple of $r_1$, we see that $r_3$ is a unit at $p$ and hence, $r_3/(1 + s_1 \delta_1 T) = \infty$ modulo $z$. Thus, $h(p) = (z,(1:0))$, which is not in the closed subscheme defined by $I(r_0)$. Therefore, $h_1^*(I(r_0))_p$ is the unit ideal in this case.  Consequently, $h_1^*(I(r_0))$ is a locally principal ideal, which allows us to conclude that $h_1$ lifts to $X(r_0)$. 

Consider $$f_2(T) = r_3 + r_1s^{\prime} \delta^{\prime}T = r_3 \left( 1 + 
\dfrac{s^{\prime} \delta^{\prime}T}{1 + s_1 \delta_1} \right).$$ This defines a homotopy 
\[
h_2\colon  \#A^1_U \to \#P^1_U; Y/Z \mapsto r_3 \left( 1 + 
\dfrac{s^{\prime} \delta^{\prime}T}{1 + s_1 \delta_1} \right).
\]
This homotopy lifts to $X(r_0)$ as the ideal $\<r_0, r_3( 1 + s^{\prime} \delta^{\prime}T(1 + s_1 \delta_1)^{-1})\>$ is principal (since $s^{\prime} \delta^{\prime}$ is nilpotent modulo $r_0/r_3 = r^{\prime}(1 + s_1\delta_1)$). We observe that $f_2(0) = r_3$ and $f_2(1) = r_2$ and thus, $h_2$ connects $\beta_{r_3}$ to $\beta_{r_2}$.  This completes the proof.
\end{proof} 

\subsection*{Proof of Theorem \ref{theorem single point}}
Let $\omega$ be a uniformizing parameter in the ring $\@O_{C,c_0}$. Then $X_{\gamma}$ is obtained from $\#P^1_U = \Proj R[Y,Z]$ by blowing up the closed subscheme corresponding to the homogeneous ideal $\< \gamma^*(\omega), Y\>$.  Write $r_0:=\gamma^*(\omega)$.  Lemma \ref{lemma easy A1-homotopy classes} and Proposition \ref{proposition S single point} give a classification of the $\#A^1$-homotopy classes of sections of $X_{\gamma} \to U$. 

By Lemma \ref{lemma easy A1-homotopy classes}, we are reduced to considering $\#A^1$-chain homotopy classes of sections of the form $\beta_r^{r_0}$, where $r \in \mathfrak m$ and $r|r_0$ but $r_0 \notdivides r$.  Set $r'=r_0/r$.  If $U = \Spec R$ is one-dimensional, then $R$ is a discrete valuation ring, and it can be immediately seen that 
\[
\sqrt{\<r\>} + \sqrt{\<r'\>} = \sqrt{\<r, r'\>}.
\]
By \cite[Proof of Theorem 5.1, Claims 1 and 2]{Balwe-Sawant-ruled}, it follows that $\@S(X)(U) = \@S^2(X)(U)$, for every smooth henselian local scheme $U$ of dimension $1$ over $k$, as desired.  

The fact that $\@S(X) \neq \@S^2(X)$ is shown by the following concrete example.  Let $\~R = k[x,y]_{\<x,y\>}^h$. Let $r_0 = x(y^2 + x)$, $r_1 = x$. Thus $r^{\prime}= y^2+x$. We have $$\sqrt{\<r_1\>} + \sqrt{\<r'\>} = \<x, y^2 + x\> = \<x,y^2\>,$$ while $\sqrt{\<r_1,r'\>} = \<x,y\>$.  We take $r_2 = x(1+y) = x + xy$. Thus $$r_2/r_1 - 1= y \notin \sqrt{\<r_1\>} + \sqrt{\<r'\>}.$$  By Proposition \ref{proposition S single point}, $\beta^{r_0}_{r_1}$ and $\beta^{r_0}_{r_2}$ are not $\#A^1$-chain homotopic.  But by \cite[Proposition 5.3, Claim 1 and Claim 2]{Balwe-Sawant-ruled}, they map to the same element in $\@S^n(X)(\Spec \~R)$, for any $n\geq 2$.  Therefore, $\@S(X)(\Spec \~R) \neq \@S^2(X)(\Spec \~R)$.  This completes the proof of Theorem \ref{theorem single point}.

\section{Proof of the main results}
\label{section general case and applications}

We begin by proving Theorem \ref{theorem main theorem intro} and Theorem \ref{theorem application intro} for nodal blowups.

\begin{theorem}
\label{theorem main theorem nodal}
Let $X \in \@N$ be a nodal blowup in the sense of Definition \ref{definition nodal blowups}.  
\begin{enumerate}[label=$(\alph*)$]
\item The natural map
\[
\@S(X)(U) \to \@S^2(X)(U) = \pi_0^{\#A^1}(X)(U)
\]
is a bijection, for every henselian local scheme $U$ over $k$ of dimension $\leq 1$.

\item $\@S(X) \nsimeq \@S^2(X)$.
\end{enumerate}

\end{theorem}
\begin{proof}
There is a sequence of morphisms 
\begin{equation}
\label{eqn blowup}
X = X_p \to \cdots X_1 \to X_0 = \#P^1_C, 
\end{equation}
where for all $i \geq 1$, $X_{i} \to X_{i-1}$ is a blowup at some node. There exist ordered pairs of non-negative integers $(m_0, n_0), (m_1,n_1), \ldots, (m_p, n_p)$, $(m_{p+1}, n_{p+1})$ with the following properties:
\begin{itemize}
\item $m_i$ and $n_i$ are coprime or each $i$; 
\item $(m_0, n_0) = (0,1)$, $(m_1, n_1) = (1,0)$;
\item For any $i$, there exist integers $j$ and $k$ such that $1 \leq j < k< i$ such that $m_k n_j - m_j n_k = 1$ and $m_i = m_j + m_k$, $n_i = n_j + n_k$;
\item $X_i = X_{i-1}[x^{m_i}/y^{n_i}]$ (in the notation of \cite[Proposition 4.2]{Balwe-Sawant-ruled}).  In other words, $X_i$ is obtained from $X_{i-1}$ by resolving indeterminacies of the rational function $x^{m_i}/y^{n_i}$.
\end{itemize}
Note that $n_i \neq 0$ for any $i \neq 1$.  We relabel these ordered pairs as $(a_0, b_0)$, $(a_1, b_1), \ldots$, $(a_{p+1},b_{p+1})$, in such a way that $a_0/b_0 = 0 < a_1/b_1 < \ldots <a_p/b_p$ and $(a_{p+1}, b_{p+1}) = (1,0)$.  Set $s_i = a_i/b_i$ for $0 \leq i \leq p$ and $s_{p+1} = \infty$.   Note that $s_p = a_p/b_p$ is always an integer.   There are exactly $p+2$ pseudo-lines on $X$, namely $\ell_{0} = \ell_{s_0}$, $\ell_{s_1}$, $\cdots$, $\ell_{s_p}$ and $\ell_{s_{p+1}} = \ell_{\infty}$. For $i \neq j$, the pseudo-lines $\ell_{s_i}$ and $\ell_{s_{j}}$ meet if and only if $j = i+1$. The node in which $\ell_{s_i}$ and $\ell_{s_{i+1}}$ meet is locally given by the ideal $\<x^{a_{i+1}}/y^{b_{i+1}}, y^{b_{i}}/x^{a_{i}}\>$. From this description, it is clear that $X$ can be obtained from $\#P^1_C$ by blowing up the ideal $\@I = \prod_{i=1}^p \<x^{a_i}, y^{b_i}\>$.  We accordingly relabel the indices in \eqref{eqn blowup} in such a way that for every $i$, $X_{i+1}$ is obtained from $X_i$ by blowing up the ideal $\<x^{a_i}, y^{b_i}\>$.  

We may assume by using elementary transformations that every pair $(a_i, b_i)$ satisfies $a_i/b_i \leq 1$ and that $(a_p,b_p) = (1,1)$.   This follows by the analogue of \cite[Proposition 5.3]{Balwe-Sawant-ruled}; we thus omit the details.  

If $\gamma\colon  U \to C$ maps the closed point of $U$ to the generic point of $C$, then $\alpha$ factors through the open subscheme $\#P^1_{C \backslash \{c_0\}}$ of $X$. Hence any two $\gamma$-morphisms $\alpha_1, \alpha_2\colon  U \to X$ are $\#A^1$-chain homotopic. On the other hand, if $\gamma$ maps the generic point of $U$ to the closed point of $C$, then any $\gamma$-morphism $\alpha\colon  U \to X$ factors through the closed fibre of the morphism $X \to C$. The closed fibre of $X \to C$ is a connected scheme, the components of which are copies of $\#P^1_k$. Thus, again, in this case too, any two morphisms $\alpha_1, \alpha_2\colon  U \to C$ which lift $\gamma$ are $\#A^1$-chain homotopic. Hence, from this point onward, we will restrict our attention to the case in which $\gamma$ maps the closed point of $U$ to the closed point of $C$ and the generic point of $U$ to the generic point of $C$. 

Let $X_{\gamma}$ denote the pullback $X \times_{C, \gamma} U$. Then $X_{\gamma}$ is the blowup of $\#P^1_U$ at the the pullback $I_{\gamma}$ of the ideal $\@I$ with respect to the morphism $\#P^1_U \to \#P^1_C$. We continue to use the variables $Y$ and $Z$ as homogeneous coordinates on $\#P^1_U = U \times_{C, \gamma} \#P^1_C$ and we will also denote the pullback of the rational function $y= Y/Z$ on $\#P^1_C$ by $y$. Set $r_0 = \gamma^*(x)$.  Then $\@I_{\gamma} = \prod_{i=1}^p I_i$ where $I_i = \<r_0^{a_i}, y^{b_i}\>$.  Applying \cite[Proposition 3.7]{Balwe-Sawant-ruled} to the blowup $X_{\gamma} \to \#P^1_U$, we see that if $\alpha_1, \alpha_2\colon  U \to \#P^1_C$ are $\gamma$-morphisms which are connected by an $\#A^1$-chain homotopy $\@H$ which lifts to $X$, then there exists $r \in R$ such that the ideals $\alpha_1^*(\@I)$, $\alpha_2^*(\@I)$ and $h_{\@H}^*(\@I)$ are all generated by $r$. The support of the ideal $\@I$ is the point $(c_0, [0:1])$. Hence, $\alpha_1$ maps the closed point of $U$ to $(c_0, [0:1])$ if and only if $\alpha_2$ does the same. Any two morphisms $\alpha_1, \alpha_2\colon  U \to \#P^1_C$ which map the closed point of $U$ to $\ell_0 \backslash \{(c_0, [0:1])\}$ are clearly $\#A^1$-chain homotopic via a homotopy which factors through the open subscheme $\#P^1_C \backslash \{(c_0, [0:1])\}$. Such a homotopy clearly lifts to $X$. Hence, we see that any two $\gamma$-morphisms of this type are $\#A^1$-chain homotopic. Hence, we will now focus on $\gamma$-morphisms $U \to \#P^1_C$ which map the closed point of $U$ to $(c_0, [0:1])$. 

For any $r \in R$, let $\beta_r\colon  U \to \#P^1_U$ be the morphism induced by the $R$-algebra homomorphism $R[y] \to R$, $y \mapsto r$. Note that any section $\beta\colon  U \to \#P^1_U$ such that $\beta(u) = (u, [0:1])$ is of the form $\beta_r$ for some unique $r \in \mathfrak{m}$. Such a section lifts to $X_{\gamma}$ if and only if the ideal $\<r_0^{a_i}, r^{b_i}\>$ is principal for every $i$, $1 \leq i \leq p$. 

By \cite[Proposition 3.7]{Balwe-Sawant-ruled}, we see that the the ideals (or ideal sheaves) $\beta_{r_1}^*(I_1)$, $\beta_{r_2}^*(I_1)$, $h_{\@H}^*(I_1)$ are all generated by a fixed element of $R$, which we denote by $r$.  we now have two cases:

\medskip

\noindent {\bf Case A:} $r_0 | r$. This happens when the ideals $\alpha_1^*(I_1)$, $\alpha_2^*(I_1)$ are generated by $r_0$. Thus $r_0|r_1$ and $r_0|r_2$. 

\noindent {\bf Case B:} $r|r_0$ but $r_0 \notdivides r$.  In this case, for $i = 1,2$, the ideal $\alpha_i^*(I_1) = \<r_0, r_i\>$ is generated by $r_i$. As $\<r_0, r_1\> = \<r_0, r_2\>$, we see that $r_1$ and $r_2$ must be unit multiples of $r$.   

\medskip

At least one of the cases A or B must hold for the lifts of both $\beta_{r_1}$ and $\beta_{r_2}$ to $X_{\gamma}$ to be $\#A^1$-chain homotopic.  

If we are in Case A, consider the homotopy 
\[
h\colon  \Spec R[T] = \#A^1_U \to \#A^1_U = \Spec R[y] 
\]   
defined by $y \mapsto r_1T + r_2(1-T)$. Then as $a_i \leq b_i$ for all $i$, and as $r_0$ divides $r_1$ and $r_2$, we get
\[
h^*(\<r_0^{a_i}, y^{b_i}\>) = \<r_0^{a_i}, (r_1T + r_2(1-T))^{b_i}\> = \<r_0^{a_i} \>    
\]
which is a principal ideal. Thus, the homotopy $h$ lifts to $X_{\gamma}$ and the lifts of $\beta_{r_1}$ and $\beta_{r_2}$ are $\#A^1$-chain homotopic. 

Now, we consider Case B. Thus, $r$ is an element of $R$ such that $r|r_0$, $r_0 \notdivides r$ and $r_1$ and $r_2$ unit elements of $\mathfrak{m}$ such that $r_i/r$ is a unit for $i = 1,2$.  Let $r_1, r_2 \in \mathfrak{m}$ be such that the sections $\beta_{r_1}$ and $\beta_{r_2}$ lift to $X_{\gamma}$ and are connected by an $\#A^1$-chain homotopy $\@H$ which also lifts to $X_{\gamma}$. As $(a_p,b_p) = (1,1)$, the sections $\beta_{r_1}$, $\beta_{r_2}$  and the $\#A^1$-chain homotopy $\@H$ lift to $X_1 \times_{U, \gamma} \#P^1_U$ which is obtained from $\#P^1_U$ by blowing up the ideal $I_1 =\<r_0, y\>$.  By Proposition \ref{proposition S single point}, $r_1/r_2$ must be a unit and we must have 
\begin{equation}
\label{equation S criterion}
r_1/r_2 -1 \in \sqrt{\<r_1\>} + \sqrt{\<r'\>} \subset \sqrt{\<r_1, r'\>},
\end{equation}
where $r' = r_0/r_1$.  Now suppose that the condition \eqref{equation S criterion} holds.  We will show that the lifts of $\beta_{r_1}$ and $\beta_{r_2}$ are $\#A^1$-chain homotopic by an argument analogous to the one used in the proof of Proposition \ref{proposition S single point}.  Let $s_1$ and $s^{\prime}$ be the square-free parts of $r_1$ and $r^{\prime}:= r_0/r_1$ and write $r_2/r_1 -1 = s_1 \delta_1 + s' \delta'$, for some $\delta_1, \delta' \in R$.  We define $r_3 = r_1(1 + s_1 \delta_1)$.  Consider the rational function $$f_1(T) = \dfrac{r_3}{1 + s_1 \delta_1 T}.$$ The ideal $\<r_3, 1 + s_1 \delta_1 T\>$ is the unit ideal since $s_1 \delta_1$ is nilpotent modulo $r_1$ and consequently, nilpotent modulo $r_3$. Thus $f_1$ defines a homotopy 
\[
h_1\colon  \#A^1_U \to \#P^1_U; Y/Z \mapsto \dfrac{r_3}{1 + s_1 \delta_1 T}.
\]
Since $f_1(0) = r_3$ and $f_1(1) = r_1$, this homotopy connects $\beta_{r_3}$ to $\beta_{r_1}$. Now, let $p$ be a point of $\#A^1_U$ which projects to $z \in U$ under the projection $\#A^1_U \to U$. We see that for every $i$, $$h_1^*(I_i)_p = \<r_0^{a_i}, \left(\frac{r_3 }{1 + s_1 \delta_1T}\right)^{b_i}\>$$ is principal if $1+s_1 \delta_1T$ is a unit at $p$. If $1 + s_1 \delta_1 T$ is not a unit, then $r_1$ is also a unit at $p$ and it follows that $h_1^*(I_i)_p$ is the unit ideal since $h_1(p) \notin Supp(I_i)$.  Consequently, $h_1^*(\@I_{\gamma})$ is a locally principal ideal, which allows us to conclude that $h_1$ lifts to $X_{\gamma}$. 

Consider $$f_2(T) = r_3 + r_1s^{\prime} \delta^{\prime}T = r_3 \left( 1 + 
\dfrac{s^{\prime} \delta^{\prime}T}{1 + s_1 \delta_1} \right).$$ This defines a homotopy 
\[
h_2\colon  \#A^1_U \to \#P^1_U; Y/Z \mapsto r_3 \left( 1 + 
\dfrac{s^{\prime} \delta^{\prime}T}{1 + s_1 \delta_1} \right).
\]
If $r_0^{a_i}|r_3^{b_i}$, then there is nothing to prove.  Assume that $r_3^{b_i}|r_0^{a_i}$ and $r_0^{a_i}\notdivides r_3^{b_i}$ at $p \in \#A^1_U$.      If $\mathfrak{p} \subset R[T]$ is the prime ideal corresponding to $p$, then  $r_0^{a_i}/r_3^{b_i} \in \mathfrak{p}$   implies that $s' \in \mathfrak{p}$.  Thus,
\[
h_2^*(I_i)_p = \< r_3^{b_i} \> \< r_0^{a_i}/r_3^{b_i}, \left( 1 + 
\dfrac{s^{\prime} \delta^{\prime}T}{1 + s_1 \delta_1} \right) \>_p 
\]
is principal.  If $s' \notin \mathfrak{p}$, then $r_0/r_3$ is a unit at $p$.  Since
\[
\dfrac{r_0^{a_i}}{r_3^{b_i}} = \left( \dfrac{r_0}{r_3}\right)^{a_i} \left( \dfrac{1}{r_3}\right)^{b_i-a_i} 
\]
is regular at $p$ and $r_0/r_3$ is a unit, $1/r_3$ is regular at $p$.  Therefore, $r_3$ and hence, $r_0$ is a unit at $p$.  Consequently, $\dfrac{r_0^{a_i}}{r_3^{b_i}}$ is a unit at $p$, showing that $h_2^*(\@I_{\gamma})$ is a locally principal ideal, which allows us to conclude that $h_2$ lifts to $X_{\gamma}$. Thus, we have proved that in Case B, the lifts of $\beta_{r_1}$ and $\beta_{r_2}$ are $\#A^1$-chain homotopic if and only if the condition \eqref{equation S criterion} holds.

The proof of part $(a)$ can be now completed exactly as in the proof of Theorem \ref{theorem single point} by observing that when $R$ is a discrete valuation ring, one has $\sqrt{\<r_1\>} + \sqrt{\<r'\>} = \sqrt{\<r_1, r'\>}$.
 
We now prove part $(b)$ by constructing a counterexample similar to the one given in the proof of Theorem \ref{theorem single point}.  For $X$ as above, consider the $2$-dimensional henselian local ring $\~R = k[x,y]^h_{\<x,y\>}$ with $U = \Spec \~R$ and consider the morphism $\gamma\colon  U \to C$  defined by $r_0 = x^{b_1}(y^2+x)$.  Consider $\beta_{r_1}, \beta_{r_2}\colon  U \to \#P^1_U$ determined by $r_1 = x$ and $r_2 = x(y+1)$.  Note that $\beta_{r_1}$ and $\beta_{r_2}$ lift to $X_{\gamma}$, since for every $i>1$ we have $a_i/b_i>a_1/b_1$.  Their lifts both map the closed point of $U$ to the point of $X_{\gamma}$ defined by the ideal $(y, x^{a_1}/y^{b_1})$, where in fact, $a_1=1$.  Now, $r' = r_0/r_1 = y^2+ x$ and we have $$\sqrt{\<r_1\>} + \sqrt{\<r'\>} = \<x, y^2 + x\> = \<x,y^2\>,$$ while $\sqrt{\<r_1,r'\>} = \<x,y\>$.  Thus $$r_2/r_1 - 1= y \notin \sqrt{\<r_1\>} + \sqrt{\<r'\>}.$$  By the criterion \eqref{equation S criterion}, we conclude that the lifts of $\beta_{r_1}$ and $\beta_{r_2}$ are not $\#A^1$-chain homotopic, but are $1$-ghost homotopic by \cite[Proof of Theorem 5.1, Claim 2]{Balwe-Sawant-ruled}.  Consequently, $\@S(X)(U) \neq \@S^2(X)(U)$.  By Remark \ref{remark sing a1-local}, it follows that $\Sing X$ is not $\#A^1$-local.
\end{proof}

We are now set to prove Theorem \ref{theorem main theorem intro} and Theorem \ref{theorem application intro} for a general smooth projective ruled surface.  Since the method used here is exactly parallel to the one used in \cite[Section 6]{Balwe-Sawant-ruled}, we will only outline the proof and refer the reader to \cite{Balwe-Sawant-ruled} for details.  

\subsection*{Proofs of Theorem \ref{theorem main theorem intro} and Theorem \ref{theorem application intro}}

Let $E$ be the projectivization of a rank $2$ vector bundle over a smooth projective curve $C$ of genus $g>0$ over an algebraically closed field $k$ of characteristic $0$. Let $X$ be a smooth projective surface with minimal model $E$.  Let $U$ be a smooth henselian local scheme over $k$.  Since $C$ is $\#A^1$-rigid, by \cite[Lemma 2.12]{Balwe-Sawant-ruled}, if two morphisms $U \to X$ are $\#A^1$-chain homotopic, then their compositions with the map $X \to C$ have to be the same. Thus, we fix a morphism $\gamma\colon  U \to C$ and characterize $\#A^1$-chain homotopy classes of morphisms $U \to X$ over $\gamma$.  Note that all the $\#A^1$-chain homotopies of $\gamma$-morphisms $U \to X$ factor through the pullback of $X$ by the morphism $\Spec \@O_{C,c_0}^h \to C$.  Note that the pullback of $E$ by the morphism $\Spec \@O_{C,c_0}^h \to C$ is just $\#P^1 \times {\Spec \@O_{C,c_0}^h}$.  We can therefore replace $C$ with $\Spec \@O_{C,c_0}^h$ and assume that $C$ is as in Notation \ref{notation nodal blowups}.  Since the case $X=E$ is clear from \cite[Proposition 2.13]{Balwe-Sawant-ruled}, we will henceforth assume that $X$ is obtained by blowing up $\#P^1_C$ at nonzero, finitely many successive blowups at smooth, closed points $\{Q_1, \ldots, Q_m\}$ and by blowing up the ideal $\@J$.  For any $Q \in \{Q_1, \ldots, Q_m\}$, points in the exceptional divisor over $Q$ are said to be \emph{infinitely near} to $Q$.  Let $N_Q$ denote the number of points in $\{Q_1, \ldots, Q_m\}$, which are infinitely near to $Q$.  Define $$N_X:= \max \{N_{Q_1}, \ldots, N_{Q_m}\}.$$ Theorem \ref{theorem main theorem intro} is proved by induction on $N_X$, which should be viewed as the \emph{complexity} of $X$ as a blowup of $\#P^1_C$.  

Fix a smooth $1$-dimensional henselian local ring $(R, \mathfrak{m})$, $U = \Spec R$ and a morphism $\gamma\colon  U \to C$.  Denote the pullback of $X$ with respect to $\gamma$ by $X_{\gamma}$.  Let $\alpha, \alpha'\colon  U \to \#P^1_C$ be morphisms over $\gamma$ that lift to $X$.  In \cite[Proof of Theorem 1.5]{Balwe-Sawant-ruled}, it is shown that we can find a smooth, projective ruled surface $X'$ such that 
\begin{itemize}
\item $X'$ is obtained by taking successive blowups of finitely many smooth, closed points of $\#P^1_C$;
\item $\~X$ is an open subscheme of $X'$;
\item $\alpha, \alpha'$ lift to $X'$ and their lifts in $X'$ are $n$-ghost homotopic if and only if the lifts of $\alpha, \alpha'$ in $\~X$ are $n$-ghost homotopic; and 
\item $N_{X'}<N_X$.
\end{itemize}
We observe that $\#A^1$-chain homotopies are nothing but $0$-ghost homotopies.  One can now finish the proof of Theorem \ref{theorem main theorem intro} using Theorem \ref{theorem main theorem nodal} and the induction hypothesis.

For the proof of Theorem \ref{theorem application intro}, observe that one can factor $X \to \#P^1_C$ as $X \to Y \to \#P^1_C$ such that $Y$ is a nodal blowup.  One can now take the counterexample given in the proof of Theorem \ref{theorem single point} and observe that it lifts to $Y_{\gamma}$ and $X_{\gamma}$ to obtain $\@S(X) \neq \@S^2(X)$.  By Remark \ref{remark sing a1-local}, it follows that $\Sing X$ is not $\#A^1$-local, when $X$ is not a minimal model. \qed

\end{document}